\documentclass[11pt]{amsart}
\usepackage{amsmath,amsthm,amssymb,amsxtra}
\usepackage{pb-diagram}
\usepackage{verbatim}
\usepackage{hyperref}
\usepackage{extpfeil}
\usepackage[margin=1.1in]{geometry}

\newcommand{\af}{{\mathrm{af}}}
\newcommand{\al}{\alpha}

\newcommand{\ap}{a_{>0}}
\newcommand{\A}{\mathbb{A}}

\newcommand{\B}{\mathbb{P}}
\newcommand{\C}{\mathbb{C}}
\newcommand{\Cinf}{\C^{(-\infty,\infty]}}
\newcommand{\core}{\mathrm{core}}
\newcommand{\Core}{\mathrm{Core}}
\newcommand{\defn}[1]{{\textbf{#1}}}

\newcommand{\Fl}{\mathrm{Fl}}
\newcommand{\For}{\mathrm{For}}

\newcommand{\Fun}{\mathrm{Fun}}

\newcommand{\hh}{\hat{h}}
\newcommand{\hK}{\hat{K}}

\newcommand{\hommap}{\kappa}
\newcommand{\hs}{\hat{s}}
\newcommand{\hS}{\hat{S}}
\newcommand{\hL}{\hat{\Lambda}}
\newcommand{\Gr}{\mathrm{Gr}}

\newcommand{\hB}{\hat{\mathbb P}}
\newcommand{\hPhi}{\hat{\Phi}}

\newcommand{\id}{\mathrm{id}}
\newcommand{\Image}{\textrm{Im}}
\newcommand{\ip}[2]{\langle #1\,,\,#2\rangle}

\newcommand{\la}{\lambda}
\newcommand{\La}{\Lambda}
\newcommand{\mt}{\tilde{m}}

\newcommand{\Q}{\mathbb{Q}}

\newcommand{\pt}{\mathrm{pt}}
\newcommand{\R}{{\mathbb R}}

\newcommand{\s}{{\mathbf{s}}}

\newcommand{\Snz}{S_{\ne0}}

\newcommand{\tF}{\tilde{F}}
\newcommand{\tS}{\tilde{S}}

\newcommand{\Y}{\mathbb{Y}}
\newcommand{\Z}{\mathbb{Z}}

\newtheorem*{statement}{Theorem}
\newtheorem{prop}{Proposition}
\newtheorem{lem}[prop]{Lemma}
\newtheorem{cor}[prop]{Corollary}
\newtheorem{lemma}[prop]{Lemma}
\newtheorem{thm}[prop]{Theorem}
\newtheorem{theorem}[prop]{Theorem}
\newtheorem{conj}[prop]{Conjecture}
\theoremstyle{definition}
\newtheorem{rem}[prop]{Remark}
\newtheorem{remark}[prop]{Remark}

\newtheorem{example}[prop]{Example}
\title[$k$-double Schur functions]{$k$-double Schur functions and equivariant (co)homology of the affine Grassmannian}
\begin{document}
\author{Thomas Lam}
\address{Department of Mathematics,
University of Michigan, 530 Church St., Ann Arbor, MI 48109 USA}
\email{tfylam@umich.edu}
 \thanks{T.L. was supported by NSF grant DMS-0901111, and by a Sloan Fellowship.}
\author{Mark Shimozono}
\address{Department of Mathematics, Virginia Tech, Blacksburg, VA 24061-0123 USA}
\email{mshimo@vt.edu}
\thanks{M.S. was supported by NSF DMS-0652641 and DMS-0652648.}

\begin{abstract} The Schubert bases of the torus-equivariant homology and cohomology rings of
the affine Grassmannian of $SL_{k+1}$ are realized by new families of symmetric functions
called $k$-double Schur functions and affine double Schur functions.
\end{abstract}
\maketitle

\tableofcontents

\section{Introduction}
Let $\Gr = \Gr_{SL_n}$ denote the affine Grassmannian of $SL_n$.  The space $\Gr$ is an ind-scheme with a distinguished stratification $\Gr = \bigsqcup X_w$ by Schubert varieties.  The torus-equivariant Schubert classes of $X_w$ form dual bases of the torus-equivariant homology $H_T(\Gr)$ and cohomology $H^T(\Gr)$.  The weak homotopy equivalence between $\Gr$ and the based loop group $\Omega SU(n)$ endows $H_T(\Gr)$ and $H^T(\Gr)$ with the structure of dual Hopf algebras.  The main aim of this article is to produce explicit polynomial representatives for the equivariant Schubert classes $[X_w]_T \in H_T(\Gr)$ and $[X_w]^T \in H^T(\Gr)$.

We first construct two dual Hopf algebras of symmetric functions $\La^{(n)}(x\|a)$ and $\hL_{(n)}(y\|a)$.  The former is a quotient of the ring $\La(x\|a)$ of double symmetric functions, and the latter is a subalgebra of the ring $\hL(y\|a)$ of dual double symmetric functions.  The ring $\La(x\|a)$ is connected to the torus-equivariant cohomology of the finite Grassmannian \cite{KT}.  The ring $\hL(y\|a)$ was defined by Molev \cite{Mol}.

We introduce new symmetric functions $\tF_\la(x\|a) \in \La^{(n)}(x\|a)$, called \defn{affine double Schur functions}, and $s^{(k)}_\la(y\|a) \in \hL_{(n)}(y\|a)$, called \defn{$k$-double Schur functions}.  When the ``equivariant parameters'' $a_i$ are specialized to 0, these symmetric functions reduce to the usual affine Schur functions \cite{LamAS}, and the $k$-Schur functions of Lapointe, Lascoux, and Morse \cite{LLM,LM}.  Affine Schur functions are a special case of affine Stanley symmetric functions, and we also give a definition of \defn{affine double Stanley symmetric functions} $\tF_w(x\|a)$.  Stanley symmetric functions were first introduced in the study of reduced factorizations in the symmetric group.  On the other hand, $k$-Schur functions were introduced in the apparently unrelated subject of Macdonald polynomial positivity.

These families $\tF_\la(x\|a)$ and $s^{(k)}_\la(y\|a)$ of symmetric functions form bases of the respective rings and are labeled by partitions whose first part is less than $n$.  There is a bijection $\la \leftrightarrow w_\la^\af$ between such partitions and the affine Grassmannian permutations $\{w \in \tS_n^0\}$, which label the Schubert varieties of $\Gr$.  Our main theorem is:

\begin{statement}\
\begin{enumerate}
\item
There are dual Hopf-isomorphisms 
\begin{align*}
\hommap: &\hL_{(n)}(y\|a) \to H_T(\Gr) \\
\hommap^*:& H^T(\Gr) \to \Lambda^{(n)}(x\|a)
\end{align*}
such that
\begin{align*}
s^{(k)}_\la(y\|a) &\longmapsto [X_{w_\la^\af}]_T \\
[X_{w_\la^\af}]^T &\longmapsto \tF_\la(x\|a).
\end{align*}
\item
For each $T$-fixed point $v \in \Gr$, there is a map $\epsilon_v: \La^{(n)}(x\|a) \to H^T(\pt)$, defined by specializing the $x$'s to $a$'s in a prescribed manner, such that
$$
\epsilon_v(\tF_\la(x\|a)) = \iota_v^*([X_{w_\la^\af}]^T)
$$
where $\iota_v: \pt \to \Gr$ is the inclusion of the $T$-fixed point. 
\end{enumerate}
\end{statement}
The analogue of (1) for the non-equivariant (co)homologies $H_*(\Gr)$ and $H^*(\Gr)$ was previously established by the first author \cite{Lam}, confirming a conjecture of the second author.  Part (2) of the theorem is an analogue of a well-known property of double Schur functions \cite{KT} and the double Schubert polynomials of Lascoux and Sch\"{u}tzenberger \cite{B} \cite{LSch}.  Thus our polynomial representatives combine the best of two worlds: (a) they are compatible with the Hopf structure, that is, they correctly calculate product and coproduct structure constants; and (b) they compute equivariant localizations by specializing one set of variables in terms of the other.  Indeed, that a family of symmetric functions with both these properties could be constructed was somewhat of a surprise.

In the $n \to \infty$ limit, one has the infinite Grassmannian $\Gr_\infty$ instead of the affine Grassmannian.  As a preliminary step in our calculations, we also connect the ring $\La(x\|a)$ to the torus-equivariant cohomology $H^{T_\Z}(\Gr_\infty)$ of the infinite Grassmannian.  In this limit, the affine double Schur functions $\tF_\la(x\|a)$ reduce to the double Schur functions $s_\la(x\|a)$, and we show that the latter can be identified with the Schubert bases of $\Gr_\infty$.

Our general strategy to establish part (1) of the Theorem mostly follows \cite{Lam}. Peterson showed \cite{P} that the Hopf algebra $H_T(\Gr)$ can be identified with a centralizer subalgebra $\B$ of a small torus variant $\A$ of Kostant and Kumar's nilHecke algebra for $\hat{SL}_n$, and furthermore characterized the image of the Schubert basis in $\B$ \cite{P}.  The affine double Schur functions (and more generally the affine double Stanley functions) are defined by declaring their expansion coefficients in Molev's double monomial symmetric function basis to be given by coefficients of the ``standard'' divided difference basis of $\A$ in products of Peterson's Schubert basis $\{j_w\} \in \B$.  An explicit (positive, combinatorial) formula for these coefficients is discussed in a separate article on the equivariant homology affine Pieri rule for $H_T(\Gr)$ \cite{LS:Pieri}.\footnote{However, this does not lead to a completely combinatorial formula for $\tF_\la(x\|a)$.  Molev's double monomial symmetric functions are defined via duality, and as far as we are aware, do not yet have an explicit combinatorial description.}  We compute the coproduct structure constants for certain special classes of $\B$ and observe that they agree with those for Molev's dual homogeneous symmetric functions.  This allows us to obtain a Hopf morphism from $\hL_{(n)}(y\|a)$ to $\B$.

To obtain part (2) of our theorem, we study the GKM ring $\Phi$ for $H^{T_\Z}(\Gr_\infty)$ and the small torus GKM ring $\Phi_\Gr$ , which is a model for $H^T(\Gr)$.  The ring $\Phi$ is a ring of functions $f: S_\Z \to H^{T_\Z}(\pt)$ on the infinite symmetric group, satisfying certain divisibility conditions described by Goresky, Kottwitz, and Macpherson \cite{GKM}.  General GKM theory, or alternatively the work of Kostant and Kumar \cite{KK}, shows that $\Phi$ can be identified with the image of $H^{T_\Z}(\Gr_\infty)$ in $\prod_{w \in S_\Z}H^{T_\Z}(\pt)$ and that one has $\Phi \simeq H^{T_\Z}(\Gr_\infty)$.  The ring $\La(x\|a)$ of double symmetric functions can be identified with $\Phi$ via a sequence of specialization homomorphisms, where the $a$-variables are substituted for the $x$-variables in prescribed ways.  While $H^T(\Gr)$ does not satisfy the original conditions required for the GKM construction \cite{GKM}, nevertheless Goresky, Kottwitz, and Macpherson still give an analogous isomorphism $H^T(\Gr) \simeq \Phi_\Gr$ in \cite{GKM2}.

The affine Grassmannian $\Gr$ can be embedded in the infinite Grassmannian $\Gr_\infty$, and this induces a maps $H^{T_\Z}(\Gr_\infty)\to H^T(\Gr_\infty)\to H^T(\Gr)$.  We explicitly calculate the kernel $I_n$ of the induced (surjective) map $\La(x\|a) \simeq H^{T_\Z}(\Gr_\infty) \to H^T(\Gr)$, obtaining a description of the quotient $\La^{(n)}(x\|a) \simeq H^T(\pt)\otimes_{H^{T_\Z}(\pt)} \La(x\|a)/I_n$.  This description is then compared to the Hopf algebras obtained via the definition of affine double Schur functions, and shown to agree.

\section{Affine and infinite Grassmannians}
We refer the reader to Kumar \cite{Ku} for more details concerning the geometry discussed in this section.
\subsection{Infinite Grassmannian}
We recall the ind-variety structure on the infinite Grassmannian (see, for example, \cite{KLMW}.
Let $\Cinf$ be the vector space containing independent elements
$e_i$ for $i\in\Z$ and consisting of infinite $\C$-linear combinations $\sum_{i\in\Z} c_i e_i$
where $c_i=0$ for $i\ll 0$. For $i\in\Z$ let $E_i=\prod_{j\ge i} \C e_j$.
Let $\Gr_\infty'$ be the set of subspaces $V\subset\Cinf$ such that
$E_a \supset V \supset E_b$ for some $a\le 0 < b$. Then the infinite Grassmannian $\Gr_\infty$
is the set of $V\in \Gr_\infty'$ such that $\dim((V+E_1)/E_1) = \dim((V+E_1)/V)$.
Note that $\{V\in\Gr_\infty\mid E_a \supset V\supset E_b\}\cong \Gr(b-1,b-a)$
via $V\mapsto V/E_b$. This gives $\Gr_\infty$ the structure of an ind-variety, being the union of Grassmannians.

The infinite Grassmannian $\Gr_\infty$ may also be thought of as a partial flag ind-variety for a Kac-Moody group. 
Consider the bi-infinite type $A$ root datum with Dynkin node set $I_\Z=\Z$,
simple bonds joining $i$ and $i+1$ for $i\in I_\Z$, and a weight lattice 
$\bigoplus_{i\in\Z} \Z a_i$. The Weyl group for the Kac-Moody flag ind-variety of this root datum
is $S_\Z$, the subgroup of the group $\hS_\Z$ of all permutations of $\Z$ consisting of the
permutations that move only finitely many elements.
Consider the maximal parabolic subgroup $\Snz\subset S_\Z$
generated by the reflections $s_i$ for $i\in I_\Z \setminus \{0\}$;
it is the subgroup that stabilizes the subset $\Z_{>0}$ of $\Z$
(and hence also $\Z_{\le0}$). Let $S_\Z^0$ denote the poset
of minimum length coset representatives of $S_\Z/\Snz$ under Bruhat order.
The infinite Grassmannian is the Kac-Moody partial flag ind-variety corresponding to this parabolic subgroup. 
As such, the infinite Grassmannian has a stratification by Schubert varieties $X_w$, labeled by $w \in S_\Z^0$. 

We describe the Schubert varieties explicitly.
There is a poset isomorphism from $S_\Z^0$ to Young's lattice $\Y$ of partitions,
denoted $w\mapsto \la=\la_w$, defined by
\begin{align}\label{E:law}
\la_i  &= w(1-i)+i-1 \qquad\text{for $i\in\Z_{>0}$.}
\end{align}
The inverse map $\la\mapsto w_\la$ is defined by
\begin{align}
\label{E:wla}
  w_\la &= \dotsm \rho^{-2+\la_3}_{-2}\rho^{-1+\la_2}_{-1} \rho^{\la_1}_0 &\qquad&\text{where} \\
\label{E:rhoij}
  \rho_i^j &= s_{j-1} s_{j-2}\dotsm s_{i+1} s_i &&\text{for $j\ge i$ in $I_\Z$.}
\end{align}
There is a bijection between $\Y$ and the set of \defn{almost natural subsets}
of $\Z$, those subsets $I\subset \Z$ such that the sets $\Z_{>0}\setminus I$ and
$I\cap \Z_{\le0}$ are both finite and have the same cardinality. The bijection
is given by $\la\mapsto I_\la:= \{1-\la'_1,2-\la'_2,\dotsc\}$ where $\la'$ is the conjugate partition
to $\la$.

For $w\in S_\Z^0$ the Schubert cell $X_w^\circ$ consists of the $V\in\Gr_\infty$
whose pivot set is $w\cdot \Z_{>0}$, where the pivot set of $V$ is defined by
$$\{i\in\Z\mid \dim((V+E_{i+1})/E_{i+1})>\dim((V+E_i)/E_i)\}.$$
The Schubert variety $X_w$ is the Zariski closure of $X_w^\circ$.

Let $T_\Z \simeq (\C^*)^{\Z}$ be the torus with a coordinate weight for each $e_i$.
Then $T_\Z$ acts on $\Cinf$ and therefore on $\Gr_\infty$.
The equivariant cohomology $H^{T_\Z}(\Gr_\infty)$ is a free $H^{T_\Z}(\pt)$-module 
with basis given by the equivariant Schubert classes $\{[X_w]^T\mid w\in S_\Z^0\}$ (see \cite{Ku}).
Each Schubert cell $X_w^\circ$ contains a unique $T$-fixed point $e_{w\cdot \Z_{>0}}$
where $e_I:=\prod_{i\in I} \C e_i$ for the almost natural subset $I\subset \Z$.

We have $H^{T_\Z}(\pt)\cong \Q[a] = \Q[\ldots,a_{-1},a_0,a_1,\ldots]$.\footnote{Cohomology is taken with coefficients in $\Q$ unless
explicitly stated otherwise.}  We shall make the nonstandard choice that the weight of $e_i$ is $-a_{1-i}$.
The reflections in $S_\Z$ are the transpositions $s_{ij}$ for integers $i<j$;
we have $s_i=s_{i,i+1}$. The associated root of $s_{ij}$ is denoted $\alpha_{ij}$. Because of our unconventional indexing of the $a_i$, \textbf{the root $\alpha_{ij}$ is identified with the element $a_{1-j} - a_{1-i}$ of the weight lattice.}

\subsection{Affine Grassmannian}
\label{SS:Gr}
Let $\Gr=\Gr_{SL_n}$ denote the affine Grassmannian of $SL_n$.  This is the 
ind-scheme $SL_n(\mathcal{K})/SL_n(\mathcal{O})$ where $\mathcal{O}=\C[[t]]$
and $\mathcal{K}=\C[[t]][t^{-1}]$. One may obtain $\Gr$ as
a partial Kac-Moody flag ind-variety for the affine type $A_{n-1}^{(1)}$
root datum with Dynkin node set $I_\af = \Z/n\Z$ and parabolic Dynkin subset
$\Z/n\Z\setminus \{0+n\Z\}$. 

Let $\tS_n$ be the affine symmetric group. It has generators $s_i$ for $i\in\Z/n\Z$.
There is an embedding $\tS_n\to \hS_\Z$ given 
by $s_{j+n\Z}\mapsto\dotsm s_{j-n}s_j s_{j+n} s_{j+2n}\dotsm$ for $j\in\Z$. The maximal parabolic
subgroup of $\tS_n$ generated by $s_i$ for $i\in\Z/n\Z\setminus\{0\}$ is isomorphic to $S_n$.  The set $\tS_n^0 \subset \tS_n$ of affine Grassmannian permutations are the minimum length coset representatives in $\tS_n/S_n$.

The affine Grassmannian has a stratification by Schubert varieties $X_w$ labeled by $w \in \tS_n^0$.  Let $T \subset GL_n$ be the maximal torus of $GL_n$, which acts on $\Gr$.  The equivariant homology $H_T(\Gr)$ and equivariant cohomology $H^T(\Gr)$ have equivariant Schubert bases $\{[X_w]_T\}$ and $\{[X_w]^T\}$.  Furthermore, the inclusion of the based polynomial loops $\Omega_{{\rm pol}} SU(n) \hookrightarrow SL_n(\mathcal{K})$ induces a homeomorphism between $\Omega_{{\rm pol}} SU(n)$ and $\Gr$.  This endows $H_T(\Gr)$ and $H^T(\Gr)$ with the structure of dual Hopf algebras over $S = H^T(\pt)$, and the corresponding equivariant Schubert bases are dual bases.  In the following we shall identify $S$ with $\Q[a_i \mid i \in \Z/n\Z]$, and furthermore identify the simple root $\alpha_i$ with $a_{-i} - a_{1-i}$.  

There is a bijection $\la\mapsto w_\la^\af$ from the set of \defn{$(n-1)$-bounded partitions},
that is, those $\la\in\Y$ with $\la_1<n$, to $\tS_n^0$,
where $w_\la^\af$ is given by the same formula as \eqref{E:wla} except that
the reflection $s_j\in S_\Z$ for $j\in\Z$ is replaced by $s_{j+n\Z}\in \tS_n$. 
There is a bijection from $\tS_n^0$ to the set $\Core_n$ of $n$-cores,
denoted $w\mapsto \core(w)=\la\in\Y$ where $w_\la\in S_\Z^0$ is the unique element
such that $w\cdot \Z_{>0} = w_\la \cdot \Z_{>0}$, viewing $w$ as a permutation of $\Z$ via
$\tS_n\to\hS_\Z$.

The affine symmetric group $\tS_n$ is the semidirect product $S_n \ltimes Q^\vee$ where $Q^\vee \simeq \Z^{n-1} \simeq \{(\nu_1,\ldots,\nu_n) \mid \nu_1 + \nu_2 + \cdots + \nu_n = 0\}$ is the coroot lattice of $S_n$. The coroots $\alpha_{ij}^\vee$ for $1 \leq i\neq j \leq n$ are identified with the vector $\nu = e_i-e_j$ where $e_k$ denotes the usual basis vector.  For $\nu \in Q^\vee$, we let $t_\nu \in \tS_n$ denote the corresponding translation element, which is given by setting $t_\nu(i) = i+\nu_i\,n$ for $1 \leq i \leq n$.   For $u \in S_n$ and $\nu \in Q^\vee$, we have the relation $ut_\nu u^{-1} = t_{u \cdot \nu}$.

The real roots of $\tS_n$ are of the form $\tilde{\al} = m\delta + \alpha$, where $\alpha$ is a root of $S_n$ and $m \in \Z$.  The corresponding reflection is given by $s_{\tilde{\al}} = s_\alpha t_{m\alpha^\vee}$.  

We shall let $\tS_n$ act on $S$ by the \defn{level $0$}-action, under which translation elements act \textit{trivially}, and $S_n$ acts in the usual manner (apart from our identification $\alpha_i = a_{-i} - a_{1-i}$). 

The affine Grassmannian $\Gr$ is naturally a sub-ind-variety of the infinite Grassmannian $\Gr_\infty$.
Let $e_1,\dotsc,e_n$ be the standard $\mathcal{K}$-basis of the vector space $\mathcal{K}^n$, 
and define $e_{i+kn} := t^k e_i$ for
$1\le i\le n$ and $k\in\Z$. This identifies $\mathcal{K}^n$ with $\Cinf$.
Then $\Gr$ consists of those $V\in \Gr_\infty$ that define 
an $\mathcal{O}$-submodule of $\mathcal{K}^n$ of rank $n$ \cite{Lu}. The $T$-fixed points in $\Gr\subset\Gr_\infty$
are the points $e_I\in \Gr_\infty$ indexed by almost natural subsets $I\subset \Z$
with the additional condition (coming from the $\mathcal{O}$-module condition) 
that if $i\in I$ then $i+kn\in I$ for all $k\ge0$.
For $w\in \tS_n$ the set $I_w := w\cdot \Z_{>0}$ is 
again almost natural and $I_{wu}=I_w$ for $u\in S_n\subset \tS_n$.

Forgetting from the $T_\Z$ action on $\Cinf$ to that of $T$, we obtain
the forgetful $\Q$-algebra homomorphism $\For:H^{T_\Z}(\pt)\to H^T(\pt)$, given by
\begin{align}\label{E:Forget}
  \For(a_{i+kn}) = a_i \qquad\text{for $1\le i\le n$ and $k\in\Z$.}
\end{align}
We shall often view $S$ as the above quotient of $\Q[a]$ without mention.


\section{GKM rings}
In this section we introduce the GKM (Goresky-Kottwitz-Macpherson) rings for the equivariant cohomology rings $H^{T_\Z}(\Gr_\infty)$ and $H^T(\Gr)$.
 
\subsection{Infinite Grassmannian and the GKM ring $\Phi$}
Let $\Fun(S_\Z,\Q[a])$ be the $\Q[a]$-algebra
of functions $S_\Z\to\Q[a]$ under pointwise product $(fg)(w)=f(w)g(w)$
and scalar action $(qf)(w)=q \,f(w)$ for $q\in \Q[a]$, $f,g\in\Fun(S_\Z,\Q[a])$,
and $w\in S_\Z$.  Let $\hPhi$ be the set of $g\in \Fun(S_\Z,\Q[a])$ such that
\begin{align}
\label{E:GKMinf}
g(s_{ij}w)-g(w)&\in \alpha_{ij}\Q[a] &\qquad&\text{for all integers $i<j$ and $w\in S_\Z$} \\
\label{E:cosetinf}
g(wu) &= g(w) &&\text{for all $w\in S_\Z$ and $u\in \Snz$.}
\end{align}
It can be shown that $\hPhi$ is a $\Q[a]$-subalgebra of $\Fun(S_\Z,\Q[a])$.  The following results follow from the work of Kostant and Kumar \cite{KK}.

\begin{prop} \label{P:GKMGrinf} 
There are (Schubert basis) elements $\xi^v\in\hPhi$ for $v\in S_\Z^0$, uniquely defined by
the properties 
\begin{align}
\label{E:xisuppinf}
\xi^v(w)&=0\qquad\text{unless $v\le w$} \\
\xi^v(w) & \in \Q[a] \, \text{is homogeneous of degree $\ell(v)$} \\
\label{E:xidiaginf}
\xi^v(v)&=\prod_{\substack{i<j \\ s_{ij}v<v}} \alpha_{ij}.
\end{align}
\end{prop}

One has $\hPhi = \prod_{v\in S_\Z^0} \Q[a] \xi^v$, and one defines
$$
\Phi = \bigoplus_{v\in S_\Z^0} \Q[a] \xi^v.
$$


For $\mu\in\Y$ let $i_\mu:\{\pt\}\to \Gr_\infty$ be the inclusion of the $T_\Z$-fixed point $e_{I_\mu}$.  If $c \in H^{T_\Z}(\Gr_\infty)$ is any equivariant cohomology class, then $i_\mu^*(c) \in \Q[a]$.
\begin{theorem}\label{T:GKM}
There is an isomorphism of $\Q[a]$-algebras
\begin{align*}
  H^{T_\Z}(\Gr_\infty) &\cong \Phi \\
\notag  c &\mapsto (w_\mu u \mapsto i_\mu^*(c)\mid \mu \in\Y, u\in \Snz) \\
\notag [X_{w_\la}]^T&\mapsto \xi^{w_\la}\qquad\text{for $\la\in \Y$}
\end{align*}
\end{theorem}

\subsection{The affine Grassmannian and GKM ring $\Phi_\Gr$}
Recall that we identify $\alpha_{ij}$ with $a_{1-j} - a_{1-i} \in S$.
Let $\Fun(\tS_n,S)$ be the $S$-algebra of functions $\tS_n\to S$ under pointwise product.  
Let $\hPhi_\Fl$ be the subset of $f\in \Fun(\tS_n,S)$ such that for all $w\in \tS_n$ we have
\begin{align}
\label{E:smalltorusaffGr}
  f((1-t_{\alpha^\vee_{ij}})^d w) &\in \alpha_{ij}^d\,S &&\text{and} \\
\label{E:smalltorusaffflags}
  f((1-t_{\alpha^\vee_{ij}})^{d-1}(1-s_{ij})w)&\in \alpha_{ij}^d\,S&&\text{for $1\le i \neq j\le n$ and $d>0$.}
\end{align}
where $f$ is formally left $S$-linear: $f(\sum_{w\in \tS_n} c_w w) := \sum_{w\in \tS_n} c_w f(w)$
for finite sums with $c_w\in S$.  These conditions were introduced in \cite{GKM2}.
Let $\hPhi_\Gr$ be the subset of $f\in \hPhi_\Fl$ such that
\begin{align}
\label{E:cosetaff}
  f(wu) &= f(w) &&\text{for $u\in S_n$}.
\end{align}
\begin{rem} \label{R:smalltorusGrass}
Functions $f\in \Fun(\tS_n,S)$ that satisfy \eqref{E:smalltorusaffGr} and \eqref{E:cosetaff}
automatically satisfy \eqref{E:smalltorusaffflags}.
\end{rem}

\begin{prop}\label{P:GKMGrAff}
There are elements $\xi^v\in\hPhi_\Gr$ (resp. $\hPhi_\Fl$) for $v \in \tS_n^0$ (resp. $v\in \tS_n$), uniquely defined by
the properties 
\begin{align}
\label{E:xisuppGraff}
\xi^v(w)&=0\qquad\text{unless $v\le w$} \\
\label{E:xihom}
\xi^v(w) &\in S \,\text{is homogeneous of degree $\ell(v)$ for all $w \in \tS_n^0$ (resp. $w \in\tS_n$} \\
\label{E:xidiagGraff}
\xi^v(v)&=\prod_{\substack{\tilde{\al} \\ s_{\tilde{\al}}v <v}} \al 
\end{align}
where the product runs over the positive real affine roots $\tilde{\al}=\pm\al_{ij}+k\delta$ with the given descent property
and $\al=\pm \al_{ij}$ is the associated finite root. Moreover 
\begin{align}
\label{E:Grassbasis}
  \hPhi_\Gr &= \prod_{v\in \tS_n^0} S\xi^v\\
\label{E:flagbasis}
  \hPhi_\Fl &= \prod_{v\in\tS_n} S \xi^v.
\end{align}
\end{prop}
\begin{proof}
Existence follows by taking the basis constructed by Kostant and Kumar \cite{KK}, which consists of functions taking values in $H^{T_\af}(\pt)$, where $T_\af$ denotes the affine torus, and specializing the functions to the finite (small) torus.  

We now prove uniqueness for the affine Grassmannian case.  Suppose $\xi, \xi' \in \hPhi_\Gr$ both satisfy the stated properties for some $v \in \tS_n^0$.  Then $\psi = \xi-\xi'$ is supported strictly above $v$ and satisfies \eqref{E:xihom}.  Suppose $\psi(w) \neq 0$ for some $w \in \tS_n^0$ satisfying $w > v$, and such that $w$ is minimal within the support of $\psi$.  Write $w = ut_\la$ for $u \in S_n$ and $\la \in Q^\vee$.  For each positive root $\alpha$ of $S_n$, we shall show that $(u\alpha)^{m_\alpha}$ divides $\psi(w)$, where for $\tilde{\al} \in\{\delta-\alpha,2\delta-\alpha,\ldots,m_\alpha \delta-\alpha\}$ one has $ws_{\tilde{\al}} < w$.  Note that since $w \in \tS_n^0$, all of the right inversions of $w$ are of this form.  Since the $(u\alpha)$ are all relatively prime, this shows that $\psi(w)$ has degree at least $\sum_{\alpha} m_\alpha = \ell(w) > \ell(v)$, contradicting the assumption.  Apply \eqref{E:smalltorusaffGr} with $\alpha_{ij} = u\alpha$  and note that by assumption
$$
\psi(t_{k(u\alpha^\vee)}w) = \psi(w t_{k\alpha^\vee}) = \psi(w t_{k\alpha^\vee} s_{\alpha}) = \psi(w s_{k\delta-\alpha^\vee}) = 0
$$
for $k = 1,2,\ldots,m_\alpha$.  Thus $\psi((1-t_{u\alpha^\vee})^{m_\alpha} w) = \psi(w)$ is divisible by $(u\alpha)^{m_\alpha}$, as claimed.

For the case of $\hPhi_\Fl$, the strategy is the same but the computation is significantly more involved.  We shall not include the complete proof here but refer the reader to the proof of \cite[Theorem 4.3]{LSS} for an analogous calculation in the $K$-theoretic setting.

Equations \eqref{E:flagbasis} and \eqref{E:Grassbasis} are proved using a similar technique; see again \cite[Theorem 4.3]{LSS}.

%
\end{proof}

Define
\begin{align} \label{E:Phiaffdef}
  \Phi_\Fl &= \bigoplus_{v\in\tS_n} S\xi^v \\
  \Phi_\Gr &= \bigoplus_{v\in\tS_n^0} S\xi^v.
\end{align}

As before, for $w\in \tS_n^0$ we write $i_{\core(w)}:\{\pt\}\to\Gr$ for the
map with image $e_{\core(w)}\in \Gr$.  The isomorphism in the following result is due to Goresky-Kottwitz-Macpherson \cite{GKM2}; the identification of the Schubert basis follows from Kostant-Kumar \cite{KK}.

\begin{theorem}\label{T:GKMaff}
There is an $S$-algebra isomorphism
\begin{align*}
  H^T(\Gr) &\cong \Phi_\Gr \\
  c &\mapsto (w u\mapsto i_{\core(w)}^*(c))&\qquad&\text{for $w\in\tS_n^0$ and $u\in S_n$} \\
  [X_v]^T&\mapsto \xi^v&&\text{for $v\in \tS_n^0$.}
\end{align*}
\end{theorem}

\begin{rem} There is an $S$-algebra isomorphism $H^T(\Fl) \cong \Phi_\Fl$ where $\Fl$ is
the affine flag ind-variety of $SL_n$.
\end{rem}

For $1\le r\le n-1$, we define $\rho^r = \rho^r_0 = w_{(r)}^\af = s_{r-1}\dotsm s_1s_0 \in\tS_n^0$.

\begin{prop} \label{P:cohomgen}
The $S$-algebra $\Phi_\Gr$ is generated over $S$ by $\xi^{\rho^r}$ for $1\le r\le n-1$.
\end{prop}
\begin{proof}
We first claim that the non-equivariant cohomology ring $H^*(\Gr)$ is generated over $\Q$ by the non-equivariant special classes $\xi^{\rho^r}_0$.  To see this, recall (for example \cite{Lam}) that $H^*(\Gr) \simeq \Lambda^{(n)} := \Lambda/\langle m_\lambda \mid \lambda_1 \geq n \rangle$, where $\Lambda$ denotes the usual ring of symmetric functions and $m_\la$ denotes the usual monomial symmetric functions.  It is clear that the power sums $p_1,p_2,\ldots,p_{n-1}$ generate $\Lambda^{(n)}$ over $\Q$, so it follows that the homogeneous symmetric functions $h_1,h_2,\ldots,h_{n-1}$ do as well.  It is shown in \cite{Lam} that $h_i$ is identified with $\xi^{\rho_r}_0$ under the isomorphism $H^*(\Gr) \simeq \Lambda^{(n)}$, giving the claim.

Suppose $\xi^w_0 = p(\xi^{\rho^1}_0,\xi^{\rho^2}_0,\ldots,\xi^{\rho^{n-1}}_0)$ for some polynomial $p$ with $\Q$-coefficients.  It is known that the structure constants in the equivariant Schubert basis $\{\xi^{w} \mid w \in \tS_n^0\}$ of $\Phi_\Gr \simeq H^T(\Gr)$ specialize via $a_i \mapsto 0$ to the structure constants in the non-equivariant Schubert basis $\{\xi^{w}_0 \mid w \in \tS_n^0\}$.  It follows from this and the obvious grading of $\Phi_\Gr$ that 
$$
p(\xi^{\rho^1},\xi^{\rho^2},\ldots,\xi^{\rho^{n-1}}) = \xi^w + \sum_{v \; : \; \ell(v) < \ell(w)}a_v\,\xi^v
$$
for some coefficients $a_v \in S$.  By induction, one obtains that each $\xi^w$ can be written as a polynomial in the $\xi^{\rho_i}$ with $S$-coefficients, proving the claim.
\end{proof}

\section{The Hopf algebra of double symmetric functions}
\label{sec:Molev}

Let $\Lambda(x\|a)$ be the Hopf algebra of \defn{double symmetric functions} 
over $\Q[a]$ \cite{Mol}.  By definition it is the polynomial algebra over $\Q[a]$ generated by the algebraically independent primitive elements
\begin{align}\label{E:primgen}
p_r[x-\ap] = \sum_{i\ge1} (x_i^r - a_i^r)\qquad\text{for $r\ge1$}
\end{align}
where $\ap=(a_1,a_2,\dotsc)$.

\subsection{Limit construction}
\label{SS:limit}
Let $P_n=\Q[a][x_1,\dotsc,x_n]^{S_n}$ be the $\Q[a]$-Hopf algebra
of symmetric polynomials in $x_1,\dotsc,x_n$. 
The coproduct is defined by declaring that the elements
\begin{align}
p_r[(x_1+x_2+\dotsm+x_n)-(a_1+a_2+\dotsm+a_n)] = \sum_{i=1}^n (x_i^r-a_i^r)
\end{align}
are primitive.

For $m<n$ let $p_m^n:P_n\to P_m$
be the $\Q[a]$-Hopf homomorphism defined by setting $x_k=a_k$ for $m<k\le n$.
Then $\La(x\|a)$ is the projective limit 
of the $P_n$ with $p_n^\infty:\La(x\|a)\to P_n$ given by setting $x_k=a_k$ for $k>n$.

\subsection{Automorphisms}
Let $\tau$ be the $\Q$-algebra automorphisms of $\Q[a]$
and of $\La(x\|a)$ given by 
\begin{align}
\label{E:tauvar}
\tau(a_i) &= a_{i+1}&\qquad&\text{for $i\in\Z$} \\
\label{E:tausym}
\tau(p_r[x-\ap]) &= p_r[x-\ap] + a_1^r&&\text{for $r\ge1$.}
\end{align}

\subsection{Double Schur functions}
Let $A$ and $B$ be countable sets of variables.
Define the homogeneous symmetric functions $h_r[A-B]$ by
\begin{align}
\label{E:h}
  \sum\limits_{r\ge0} t^r h_r[A-B] &= \dfrac{\prod_{b\in B} 1-bt}{\prod_{a\in A} 1-at}.
\end{align}
Let $\Y$ be Young's lattice of partitions.
The \defn{double complete symmetric functions} $\{h_\la(x\|a)\mid \la\in\Y\}$ are the
$\Q[a]$-basis of $\La(x\|a)$ defined by
\begin{align}
\label{E:doublehone}
  h_r(x\|a) &= \tau^{1-r} h_r[x-\ap] \\
\notag
  &= h_r[x-\ap-a_0-a_1-\dotsm-a_{2-r}] \\
\label{E:doubleh}
  h_\la(x\|a) &= h_{\la_1}(x\|a) h_{\la_2}(x\|a)\dotsm.
\end{align}
$\La(x\|a)$ has $\Q[a]$-bases given by the 
\defn{double Schur functions} $\{s_\la(x\|a)\mid \la\in\Y\}$
and \defn{Schur functions} $\{s_\la[x-\ap]\mid \la\in\Y\}$:
\begin{align}
  s_\la(x\|a) &= \det (\tau^{j-1} h_{\la_i-i+j}(x\| a))_{1\le i,j\le \ell(\la)} \\
  s_\la[x-\ap] &= \det (h_{\la_i-i+j}[x-\ap])_{1\le i,j\le \ell(\la)}.
\end{align}

It can be shown \cite[Thm 3.20]{Mol} that 
\begin{align}\label{E:doubleSchurtoSchur}
  s_\la(x\|a) \in s_\la[x-\ap] + \sum_{\mu\subsetneq\la} \Z[a_{2-\la_1},\dotsc,a_{\ell(\la)-1}] s_\mu[x-\ap].
\end{align}

\subsection{Symmetric function ring for $\Phi$}
\label{SS:symlocinf}
Let $w\in S_\Z$. There is a $\Q[a]$-algebra homomorphism $\epsilon_w:\La(x\|a)\to \Q[a]$ defined by
the substitution $x_i\mapsto a_{1-w(1-i)}$ for all $i>0$, or more formally,
\begin{align}\label{E:epsilonpower}
  \epsilon_w(p_r[x-\ap]) = \sum_{i\in \Z_{\le0}\cap (1-w\cdot\Z_{\leq 0})} a_{i}^r - \sum_{i\in \Z_{>0}\setminus (1-w\cdot\Z_{\leq 0})} a_{i}^r.
\end{align}

Define the map $\epsilon: \La(x\|a) \to \Fun(S_\Z,\Q[a])$ defined by
\begin{align}\label{E:epsilon}
  \epsilon(f)(w) = \epsilon_w(f).
\end{align}

\begin{thm} \label{T:thminfloc}
$\epsilon$ is a $\Q[a]$-algebra isomorphism
$\La(x\|a) \cong \Phi$ sending $s_\la(x\|a)$ to $\xi^{w_\la}$.
\end{thm}

\begin{proof} $\epsilon$ defines a $\Q[a]$-algebra homomorphism $\La(x\|a) \to \Fun(S_\Z,\Q[a])$. 
Since $\Phi$ is a $\Q[a]$-subalgebra of $\Fun(S_\Z,\Q[a])$,
to show that $\epsilon$ has image in $\Phi$ it suffices to check it on $\Q[a]$-algebra generators:
for $r\ge 1$ we must show that $\epsilon(p_r[x-\ap])\in\Phi$. By degree considerations it is enough to
show that $\epsilon(p_r[x-\ap])\in\hPhi$. To show this it suffices to show 
$g\in\hPhi$ where $g(w)=\epsilon_w(p_r[x-\ap])$.
It is clear that $g$ satisfies \eqref{E:cosetinf}. Condition \eqref{E:GKMinf} is straightforward to verify
using \eqref{E:epsilonpower}.

We deduce that the image of $\epsilon$ is in $\Phi$.
By the interpolation characterization of the double Schur functions \cite{O} (see also \cite[(2.21)]{Mol}) one has
$$
\epsilon(s_\la(x\|a))(w_\la) = \prod_{(i,j) \in \la} (a_{i-\la_i} - a_{\la'_j-j+1}) 
= \prod_{(i,j) \in \la} \alpha_{w(j),w(1-i)}
= \prod_{\substack{k<\ell \\ s_{k\ell}w < w}} \alpha_{k,\ell}
$$
and in addition $\epsilon(s_\la(x\|a))(w_\mu) = 0$ unless $\la \subseteq \mu$. 
It follows that $\epsilon(s_\la(x\|a))$ satisfies the defining property of the
basis element $\xi^{w_{\la}}$. In particular $\epsilon$ is onto.
Graded dimension counting shows that $\epsilon$ is an isomorphism.
\end{proof}

\subsection{Symmetric function ring for $\Phi_\Gr$}
\label{SS:symlocaff}

Define the map $\epsilon_\Gr: \La(x\|a) \to \Fun(\tS_n,S)$ by
\begin{align}\label{E:epsilonGr}
  \epsilon_\Gr(f)(w) = \epsilon_w(f)\qquad\text{for $w\in\tS_n$}
\end{align}
where the definition for $\epsilon_w$ in \eqref{E:epsilonpower} also makes sense for $w\in \tS_n$,
and $\epsilon_w(f)$ is considered an element of $S$ via the forgetful map \eqref{E:Forget}.  From now on without additional mention, we shall view an $S$-algebra as a $\Q[a]$-algebra via
the ring homomorphism $\For$. 

\begin{lem}\label{L:epsilonGr} The image of $\epsilon_\Gr$ is contained in $\Phi_\Gr$.
\end{lem}
\begin{proof} $\epsilon_\Gr$ is a well-defined $\Q[a]$-algebra
homomorphism. Since $\Phi_\Gr$ is a $\Q[a]$-subalgebra
of $\Fun(\tS_n,S)$, it suffices to consider the images of the algebra generators: for $r\ge1$ we must show that $\epsilon_\Gr(p_r[x-\ap])\in \Phi_\Gr$.  Let $w\in \tS_n$. 

The condition \eqref{E:cosetaff} is immediate. It follows that
we may instead evaluate at $t_\nu$ such that 
$\nu=(\nu_1,\dotsc,\nu_n)\in Q^\vee$ satisfies $w S_n = t_\nu S_n$.  Since
\begin{align*}
|\Z_{\le0} \cap t_\nu\cdot \Z_{>0} \cap (i+n\Z) | &= \min(0,-\nu_i) \\
| (\Z_{>0} \setminus t_\nu\cdot\Z_{>0})\cap (i+n\Z) | &= \max(0,\nu_i) \qquad\text{for $1\le i\le n$}
\end{align*}
we have
\begin{align*}
  \epsilon_\Gr(p_r[x-\ap])(t_\nu)
  &= \left(\sum_{i\in \Z_{\le0} \cap t_\nu\cdot \Z_{>0}} a_{1-i}^r -\sum_{i\in \Z_{>0}\setminus t_\nu\cdot \Z_{>0}} a_{1-i}^r  \right) \\ \notag
  &= \sum_{i=1}^n \nu_i\, a_{1-i}^r.
\end{align*}
Denote $g(t_\nu u) := \epsilon_\Gr(p_r[x-\ap])(t_\nu u) = \sum_{i=1}^n \nu_i a_{1-i}^r$, where $u \in S_n$.  Note that the function $\nu\mapsto g(t_\nu)$ is linear in $\nu$.

To establish \eqref{E:smalltorusaffGr} let $d>0$, $\al_{ij}$ a finite root and $\nu\in Q^\vee$.
We have
\begin{align*}
g((1-t_{\al^\vee_{ij}})^d t_\nu) 
&= \sum_{p+q=d} (-1)^q \binom{d}{q} g(t_{q\al^\vee_{ij}} t_\nu) \\
&= g(t_{\al^\vee_{ij}}) \sum_{p+q=d} (-1)^q \binom{d}{q} q +  g(t_\nu) \sum_{p+q=d} (-1)^q \binom{d}{q} \\
&= (a_{1-i}^r-a_{1-j}^r) \sum_{q=0}^d (-1)^q \binom{d}{q} q.
\end{align*}
If $d=1$ then $\sum_{q=0}^d (-1)^q \binom{d}{q} q = -1$ and
$a_{1-i}^r-a_{1-j}^r$ is divisible by $a_{1-j}-a_{1-i}$. Otherwise $d\ge2$,
in which case $\sum_{q=0}^d (-1)^q \binom{d}{q} q = 0$ and
again we are done.

Therefore $\epsilon_\Gr$ well-defines a $\Q[a]$-algebra homorphism
into $\Phi_\Gr$.
\end{proof}

To describe the kernel of $\epsilon_\Gr$ we introduce another basis of $\La(x\|a)$.
It is a kind of  power sum basis. For $r\ge1$ define $\mt_r(x\|a)\in \La(x\|a)$ by
\begin{align}\label{E:mt}
  \mt_r(x\|a) = \sum_{j=1}^r (-1)^{r-j} e_{r-j}(a_1,a_0,a_{-1},\dotsc,a_{2-r}) p_j[x-\ap]
\end{align}
and define $\mt_\la(x\|a) =\mt_{\la_1}(x\|a)\mt_{\la_2}(x\|a)\dotsm$. Let $\mt_0(x\|a)=1$.
The $\mt_\la(x\|a)$ form a basis of $\La(x\|a)$, as they 
are unitriangular with the $p_\la[x-\ap]$ basis.

\begin{lem}\label{L:epsilonGrker} The kernel of $\epsilon_\Gr$ contains
$\mt_r(x\|a)$ for $r\ge n$.
\end{lem}
\begin{proof}
Let $r\ge n$. Let $\nu=(\nu_1,\dotsc,\nu_n)\in Q^\vee$
be arbitrary. It suffices to show that
$\epsilon_\Gr(\mt_r(x\|a))(t_\nu)=0$ for all $\nu\in Q^\vee$. We have
\begin{align*}
& \quad\epsilon_\Gr(\mt_r(x\|a))(t_\nu) \\
&= \sum_{k=1}^r (-1)^{r-k} e_{r-k}(a_1,a_0,a_{-1},\dotsc,a_{2-r})
 \epsilon_\Gr(p_k[x-\ap])(t_\nu)) \\
 &= \sum_{k=1}^r (-1)^{r-k} e_{r-k}(a_1,a_0,a_{-1},\dotsc,a_{2-r}) \sum_{i=1}^n \nu_i a_{1-i}^k \\
 &= \sum_{i=1}^n \nu_i (E_r(a_{1-i})-a_1a_0\dotsm a_{2-r})) \qquad\text{where} \\ 
 E_r(t) &= (t-a_1)(t-a_0)\dotsm (t-a_{2-r}).
\end{align*}
But $r\ge n$ so $E_r(a_{1-i})=0$ in $S$ for all $1\le i\le n$ and $\sum_{i=1}^n \nu_i=0$.
We conclude that $\mt_r(x\|a)\in \ker(\epsilon_\Gr)$.
\end{proof}

\begin{lem} \label{L:smallclasses} For $\la\in\Y$ such that $\la_1+\ell(\la)\le n$,
\begin{align}\label{E:doubleSchurSchubclass}
  \epsilon_\Gr(s_\la(x\|a)) = \xi^{w_\la^\af}.
\end{align}
In particular for $1\le r\le n-1$
\begin{align}\label{E:onerowclasses}
  \epsilon_\Gr(h_r(x\|a)) = \xi^{\rho^r}.
\end{align}
\end{lem}
\begin{proof} We have $\epsilon(s_\la(x\|a))=\xi^{w_\la}$.
For such partitions $\la$, $w_\la^\af$ and $w_\la$ have the same reduced words, and for $w_\la^\af$ one of the simple generators is not used.  The Bruhat order ideals $\{v \in S_\Z^0 \mid v < w_\la\}$ and $\{u \in \tS_n^0 \mid u < w_\la^\af\}$ are naturally isomorphic via a bijection $u = \gamma(v)$.    Furthermore, we have $\la_v = \core(\gamma(v))$, so that $\epsilon_v(s_\la(x\|a)) = \epsilon_{\gamma(v)}(s_\la(x\|a))$.  Since $\epsilon(s_\la(x\|a))$ satisfies Proposition \ref{P:GKMGrinf}, we also have that $\epsilon_\Gr(s_\la(x\|a))$ satisfies Proposition \ref{P:GKMGrAff}.
\end{proof}

Define
\begin{align}
\label{E:In} I_n &= \sum_{r\ge n} \La(x\|a) \mt_r(x\|a) \\
\label{E:Laupper} \La^{(n)}(x\|a) &= S \otimes_{\Q[a]} \La(x\|a)/I_n.
\end{align}
Since the $\mt_\la(x\|a)$ are unitriangular with the $p_\la[x-\ap]$
it follows that
\begin{align}\label{E:Imtbasis}
  I_n = \bigoplus_{\la_1\ge n} \Q[a] \mt_\la(x\|a).
\end{align}

\begin{thm} \label{T:CohomGr}
The map $\epsilon_\Gr$ induces an $S$-algebra isomorphism
$$\La^{(n)}(x\|a) \cong \Phi_\Gr.$$
\end{thm}
\begin{proof} There is a well-defined $S$-algebra homomorphism 
$\La^{(n)}(x\|a) \to \Phi_\Gr$ by Lemmata \ref{L:epsilonGr} and \ref{L:epsilonGrker}.
It is surjective by Lemma \ref{L:smallclasses} and Proposition \ref{P:cohomgen}.
Graded dimension counting completes the proof.
\end{proof}

\section{The ring dual to the double symmetric functions}
Let $\hL(y\|a)$ be the Hopf algebra over $\Q[a]$ given by 
the symmetric series in the variables $y=(y_1,y_2,\dotsc)$ with coefficients in $\Q[a]$.
More precisely, $\hL(y\|a)$ is the formal power series ring over $\Q[a]$ in algebraically independent
primitive elements $p_r[y]=\sum_{i\ge1} y_i^r$ for $r>0$.  We refer the reader to \cite{Mol} for more details concerning $\hL(y\|a)$.

\subsection{Pairing with double symmetric functions}
There is a $\Q[a]$-bilinear perfect pairing $\ip{}{}: \La(x\|a)\times \hL(y\|a)\to\Q[a]$
defined by either of the following, for $\la,\mu\in\Y$:
\begin{align}
\label{E:schurpair}
\ip{s_\la[x-\ap]}{s_\mu[y]} &= \delta_{\la\mu} \\
\label{E:powerpair}
\ip{p_\la[x-\ap]}{p_\mu[y]} &= z_\la \delta_{\la\mu}
\end{align}
where $z_\la = \prod_{i\ge1} i^{m_i(\la)} m_i(\la)!$ and $m_i(\la)$ is the number of times the part $i$ occurs in $\la$.

\begin{lemma}\label{L:Hopfduality}
Under this pairing, the coalgebras $\Lambda(x\|a)$ and $\hL(y\|a)$ are Hopf-dual; namely,
$$
\ip{\Delta f}{g \otimes h} = \ip{f}{gh} \qquad \ip{g \otimes h}{\Delta f} = \ip{gh}{f}.
$$
\end{lemma}

%

\subsection{Dual Schur functions}
The \defn{dual Schur functions} are the ``basis'' $\hs_\la(y\|a)$ of $\hL(y\|a)$ dual to
the double Schur basis of $\La(x\|a)$:
\begin{align}\label{E:dualSchur}
\ip{s_\la(x\|a)}{\hs_\la(y\|a)} &= \delta_{\la\mu}.
\end{align}
\begin{remark} The dual Schur functions generally have an infinite expansion in terms of
Schur functions, with $s_\la(y\|a) = s_\la(y) + \text{higher degree terms}$ \cite{Mol};
see \eqref{E:hhtos} for the case when $\la$ is a single row.
\end{remark}

The \defn{dual homogeneous symmetric functions} are given by $\hh_k(y\|a) = \hs_k(y\|a)$ for $k\in \Z_{\ge0}$
and $\hh_\la(y\|a) = \hh_{\la_1}(y\|a)\hh_{\la_2}(y\|a)\dotsm$ for $\la\in\Y$.
More explicitly, we have
\begin{align}\label{E:hhGF}
  \sum_{k\ge0} \hh_k(y\|a) (t-a_1)(t-a_0)(t-a_{-1})\dotsm(t-a_{2-k}) &= \prod_{i\ge1} \dfrac{1-a_1 y_i}{1-t y_i}
\end{align}

\begin{rem}\label{R:ha} For $k>0$, $\hh_k(y\|a)$ only depends on the $k+1$ parameters $a_1$, $a_0$,
$a_{-1}$, $\dotsc$, $a_{1-k}$.
This follows by induction upon the substitution of $t=a_{1-k}$ into \eqref{E:hhGF}:
\begin{align}\label{E:hhGFsubs}
\sum_{j=0}^k \hh_j(y\|a) (a_{1-k}-a_1)(a_{1-k}-a_0)\dots(a_{1-k}-a_{2-j}) =
\prod_{i\ge1} \dfrac{1-a_1y_i}{1-a_{1-k}y_i}.
\end{align}
\end{rem}

The dual Schur functions satisfy the Jacobi-Trudi formula 
\begin{align*}
\hs_\la(y\|a) = \det (\tau^{j-1} \hh_{\la_i-i+j}(y\|a))_{1\le i,j\le \ell(\la)}
\end{align*}

\subsection{Orthogonality}

The \defn{double monomial symmetric functions} $m_\la(x\|a)$ are by definition
the basis of $\La(x\|a)$ dual to the dual $h$-``basis'' of $\hL(y\|a)$:
\begin{align}\label{E:doublem}
\ip{m_\la(x\|a)}{\hh_\mu(y\|a)} = \delta_{\la,\mu}\qquad\text{for $\la,\mu\in\Y$.}
\end{align}

\begin{rem} It will be shown in Proposition \ref{P:monomial}
that $\mt_r(x\|a)=m_r(x\|a)$ for $r\ge1$.
However $\mt_\la(x\|a)$ and $m_\la(x\|a)$ disagree in general.
\end{rem}

\begin{lem}\label{L:mhortho} For all $r,k\ge0$ we have
\begin{align}\label{E:mhonepart}
\ip{\mt_r(x\|a)}{\hh_k(y\|a)} = \delta_{r,k}.
\end{align}
In particular, for all $r>0$ and $\mu\in\Y$ with $\mu_1 < r$ we have
\begin{align}\label{E:orthopart}
\ip{\mt_r(x\|a)}{\hh_\mu(y\|a)} = 0.
\end{align}
\end{lem}
\begin{proof} Since $\mt_r(x\|a)$ is primitive, by Lemma \ref{L:Hopfduality}
the proof of \eqref{E:orthopart} reduces to that of \eqref{E:mhonepart}.
For $k=0$, $\hh_0(y\|a)=1$ and the result holds by \eqref{E:powerpair}. 
Suppose $k>0$. For all $j>0$ we have
\begin{align}\label{E:singleptos}
  p_j[x-\ap] &= \sum_{i=0}^{j-1} (-1)^i s_{(j-i,1^i)}[x-\ap].
\end{align}
By \cite[Cor. 3.13]{Mol} we have
\begin{align}\label{E:hhtos}
  \hh_k(y\|a) &= \sum_{p,q\ge0}(-a_1)^q h_p(a_0,a_{-1},\dotsc,a_{1-k}) s_{(k+p,1^q)}[y].
\end{align}
By \eqref{E:schurpair} we have
\begin{align*}
&\,\ip{\mt_r(x\|a)}{\hh_k(y\|a)} \\
&= \sum_{j=1}^r (-1)^{r-j} e_{r-j}(a_1,a_0,a_{-1},\dotsc,a_{2-r}) \sum_{i=0}^{j-1} a_1^i h_{j-i-k}(a_0,a_{-1},\dotsc,a_{1-k}) \\
&= \sum_{j=k}^r (-1)^{r-j} e_{r-j}(a_1,a_0,a_{-1},\dotsc,a_{2-r}) h_{j-k}(a_1,a_0,a_{-1},\dotsc,a_{1-k}).
\end{align*}
If $k\ge r$ then the sum is $\delta_{kr}$ as required.
If $k<r$ the sum is $0$, since it is the coefficient of $t^{r-k}$ in the following polynomial of degree $r-k-1$:
\begin{align*}
\dfrac{(1-a_1t)(1-a_0t)(1-a_{-1}t)\dotsm(1-a_{2-r}t)}
{(1-a_1t)(1-a_0t)(1-a_{-1}t)\dotsm (1-a_{1-k}t)} = (1-a_{-k}t)\dotsm(1-a_{2-r}t).
\end{align*}
\end{proof}

Let $L_{\la\mu}\in\Q[a]$ be defined by
\begin{align}
  \mt_\la(x\|a) &= \sum_\mu L_{\la\mu} m_\mu(x\|a).
\end{align}

\begin{prop}\label{P:monomial} $L_{\la\mu}=0$ unless $\mu$ is a refinement of $\la$ (that is, $\la$
is obtained from $\mu$ by replacing each part $\mu_i$ of $\mu$ by a collection of positive integers
that sums to $\mu_i$). Moreover $L_{\la\la} = \prod_i m_i(\la)!$ and $\mt_r(x\|a)=m_r(x\|a)$.
\end{prop}
\begin{proof} This follows using Lemmata \ref{L:Hopfduality} and \ref{L:mhortho}
and the fact that the $\mt_r(x\|a)$ are primitive.
\end{proof}
Define $M_{\la\mu}\in\Q[a]$ by
\begin{align}\label{E:M}
  m_\la(x\|a) &= \sum_{\mu\in\Y} M_{\la\mu} m_\mu[x-\ap]
\end{align}
where $m_\mu[x-\ap]$ is the basis of $\La(x\|a)$ dual to the 
homogeneous symmetric functions $h_\mu[y]\in \hL(y\|a)$.

\begin{prop} \label{P:mm} $M_{\la\mu}\in \Z[a]$, $M_{\la\la}=1$,
and for $\mu \neq \la$, we have $M_{\la\mu}=0$ unless $|\mu|<|\la|$.
\end{prop}
\begin{proof} By duality, the coefficient of $m_\mu[x-\ap]$ in $m_\la(x\|a)$
is equal to the coefficient of $\hh_\la(y\|a)$ in $h_\mu[y]$. By \eqref{E:hhtos} we have
$\hh_r(y\|a) \in h_r[y] + \prod_{|\nu|>r} \Z[a] s_\nu[y]=h_r[y]+\prod_{|\nu|>r} \Z[a]h_\nu[y]$. It follows that
$\hh_\la(y\|a) \in h_\la[y] + \prod_{|\nu|>|\la|} \Z[a] h_\nu[y]$. In particular $\hh_\la(y\|a)$ appears in
$h_\la[y]$ with coefficient $1$ and does not appear in $h_\mu[y]$ unless $|\mu|<|\la|$.
\end{proof}

\begin{conj} $M_{\la\mu}=0$ unless $\mu\subset\la$. Moreover, in this case
$M_{\la\mu}\ne0$ unless $\mu=\emptyset$.
\end{conj}

\section{Equivariant Homology of $\Gr$ via symmetric functions}
\label{sec:affine}

\subsection{The affine nilHecke algebra and the Peterson subalgebra}
Let $Q=\mathrm{Frac}(S)$ be the fraction field of $S$. Let $\A_Q$ be the twisted group
algebra given by the formal $Q$-linear combinations of elements of $\tS_n$
subject to the relations $w q = (w \cdot q) w$ for $q\in Q$ and $w\in \tS_n$.
$\A_Q$ acts on $Q$ where $\tS_n$ acts on $Q$ by the level zero action and $Q$
acts on itself by left multiplication. Define the elements $A_i = \alpha_i^{-1}(s_i-1)\in \A_Q$
for $i\in\Z/n\Z$. They satisfy only the relations $A_i^2=0$ and 
the same braid relations as simple reflections $s_i$ in $\tS_n$.
The \defn{affine nilCoxeter ring} $\A_0$ is
the $\Q$-subalgebra of $\A_Q$ generated by the $A_i$.
For $w\in \tS_n$ let $A_w = A_{i_1}\dotsm A_{i_N}$
where $w=s_{i_1}\dotsm s_{i_N}$ is a reduced decomposition. Then $\A_0=\bigoplus_{w\in\tS_n} \Q A_w$.

The algebra $\A_0$ acts on $S$ and $S$ acts on itself by left multiplication.
The (small torus) affine nilHecke ring $\A$ is the
$\Q$-subalgebra of $\A_Q$ generated by $\A_0$ and $S$. 
It has relations
$A_i f = A_i \cdot f + (s_i\cdot f) A_i$ for $i\in I$
and $f\in S$. One may show that $\A = \bigoplus_{w\in \tS_n} S A_w$.

The functions $\xi^v\in\Phi_\Fl$
of Proposition \ref{P:GKMGrAff} may be expressed by the following system of equations in $\A$ \cite{KK}:
\begin{align} \label{E:xiA}
  w = \sum_{v\in \tS_n} \xi^v(w) A_v\qquad\text{for $w\in\tS_n$.}
\end{align}
\begin{example} $s_i=1+\al_i A_i=A_{\id}+\al_i A_{s_i}$. Thus $\xi^{\id}(s_i)=1$, $\xi^{s_i}(s_i)=\alpha_i$,
and $\xi^v(s_i)=0$ for $ v\not\in \{\id,s_i\}$.
\end{example}

In other words, there is an $S$-bilinear perfect pairing $\Phi_\Fl\times \A \to S$
given by $\ip{f}{a}=f(a)$, with respect to which the bases
$\{\xi^v\mid v\in\tS_n\}$ and $\{A_v\mid v\in\tS_n \}$ are dual.
$f$ is understood to be a left $Q$-linear map $f:\A_Q\to Q$ 
in the sense that if $a=\sum_w q_w w$ for $q_w\in Q=\mathrm{Frac}(S)$ then $f(a)=\sum_w q_w f(w)$.

Let $\A \otimes_S\A$ be the $S$-module given by the quotient of $\A \otimes_{\Q} \A$
by the elements $sa \otimes a' - a \otimes s a'$ for $s\in S$ and $a,a'\in \A$.
Define the function $\Delta:\A \to \A \otimes_S \A$ by
\begin{align}
\Delta(A_i) &= A_i \otimes 1 + 1 \otimes A_i + \alpha_i A_i \otimes A_i \\
\Delta(A_w) &= \Delta(A_{i_1}) \dotsm \Delta(A_{i_N}) \qquad\text{where} \\
\notag
w &= s_{i_1}\dotsm s_{i_N} \qquad\text{is a reduced decomposition} \\
\Delta(\sum_w c_w A_w) &= \sum_w c_w \Delta(A_w)\qquad\text{for $c_w\in S'$.}
\end{align}
Here the product in $\Image(\Delta)$ is defined by
\begin{align*}
  (\sum_i a_i \otimes b_i)(\sum_j c_j \otimes d_j) = \sum_{i,j} a_i c_j \otimes b_i d_j.
\end{align*}
This product is ill-defined on all of $\A \otimes_S\A$.

$\Delta$ is left $S$-linear and a ring homomorphism $\A \to \Image(\Delta)$.

Let $\B=Z_{\A}(S)$, the centralizer subalgebra, called the \defn{Peterson subalgebra}.
\begin{theorem}[\cite{P}, see also \cite{Lam}] \label{T:Pet}
\
\begin{enumerate}
\item
There is a unique left $S$-basis $\{j_w\mid w\in \tS_n^0\}$ of $\B$
with the property that for every $w\in\tS_n^0$,
$j_w = A_w + \sum_{x\in\tS_n\setminus \tS_n^0} j_w^x A_x$
for some $j_w^x\in S$. 
\item
There is an isomorphism of $S$-Hopf algebras $H_T(\Gr) \cong \B$ such that
$j_w$ is the image of the Schubert basis element $[X_w]_T$ for $w\in \tS_n^0$.
\end{enumerate}
\end{theorem}

The Hopf-duality between $H^T(\Gr)$ and $H_T(\Gr)$ induces a Hopf-duality between $\Phi_\Gr$ and $\B$,
with respect to which $\{\xi^v\mid v\in \tS_n^0\}$ and $\{j_v\mid v\in\tS_n^0\}$
are dual bases. The pairing $\Phi_\Fl\times \A\to S$ restricts to 
a perfect pairing $\Phi_\Gr\times \B\to S$. For $v,w\in\tS_n^0$ we have
\begin{align}\label{E:xijdual}
  \ip{\xi^v}{ j_w} &= \delta_{vw}
\end{align}
by \eqref{E:xiA} and the characterization of the $j_w$ basis in Theorem \ref{T:Pet}.

By special classes for $\B$ we mean the elements $j_i = j_{\rho^i} \in \B$.  Define the elements $\tilde{c}_i^{r\ell}(a)\in S$ and $c_i^{r\ell}(a)\in\Q[a]$ by
\begin{align}
  \Delta(j_i) &= \sum_{r,\ell\ge0} \tilde{c}_i^{r\ell}(a) \,j_r \otimes j_\ell \\
  \Delta \hh_i(x\|a) &= \sum_{r,\ell\ge0} c_i^{r\ell}(a) \; \hh_r(x\|a) \otimes \hh_\ell(x\|a).
\end{align}
By Remark \ref{R:ha}, we have $c_i^{r\ell}(a) \in \Q[a_1,a_0,a_{-1},\dotsc,a_{2-n}]$.
\begin{prop}\label{P:generators} The elements
$\{\hh_1(y\|a),\hh_2(y\|a),\ldots,\hh_{n-1}(y\|a)\}$ are algebraically independent over $\Q[a_1,a_0,a_{-1},\dotsc,a_{2-n}]$.
Moreover
\begin{align}\label{E:coprodconsts}
  c_i^{r\ell}(a) = \tilde{c}_i^{r\ell}(a).
\end{align}
where $c_i^{r\ell}(a)$ is considered as an element of $S$.
\end{prop}
\begin{proof}
The ``algebraic independence'' statement follows from Molev \cite{Mol}.
We have the following coefficients:
\begin{enumerate}
\item  $c_i^{r\ell}(a)$.
\item The coefficient of $h_i(x\|a)$ in $h_r(x\|a)h_\ell(x\|a)$.
\item The coefficient of $\xi^{\rho^i}$ in $\xi^{\rho^r}\xi^{\rho^\ell}$ in $\Phi_\Gr$.
\item $\tilde{c}_i^{r\ell}(a)$.
\end{enumerate}
(1) and (2) are equal by Lemma \ref{L:Hopfduality}.
By Theorem \ref{T:CohomGr} and \eqref{E:onerowclasses} (2) and (3) agree.
The Hopf duality of $\B$ and $\Phi_\Gr$ and \eqref{E:xijdual} imply that (3) and (4) agree.
\end{proof}

\subsection{Homology symmetric function ring}
Define the following subrings of $\hL(y\|a)$:
\begin{align*}
\La_{(n)}(y\|a)&=\Q[a_1,a_0,a_{-1},\dotsc,a_{2-n}][\hh_1(y\|a),\hh_2(y\|a),\dotsc,\hh_{n-1}(y\|a)] \\
\hL'_{(n)}(y\|a)&=\Q[a_1,a_0,a_{-1},\dotsc,a_{2-n}][[\hh_1(y\|a),\hh_2(y\|a),\dotsc,\hh_{n-1}(y\|a)]].
\end{align*}
$\La_{(n)}(y\|a)$ is a $\Q[a_1,a_0,a_{-1},\dotsc,a_{2-n}]$-Hopf subalgebra of $\hL(y\|a)$,
as it is clear from the earlier discussion that $\La_{(n)}(y\|a)$ is indeed closed under the coproduct.
We shall consider $\La_{(n)}(y\|a)$ as a Hopf algebra over $S$ in the obvious manner.
$\hL'_{(n)}(y\|a)$ is a completion of $\La_{(n)}(y\|a)$.

\begin{prop}\label{P:mbasis} $\La^{(n)}(x\|a)$ and $\La_{(n)}(y\|a)$
are dual Hopf algebras over $S$ with dual $S$-bases given by 
$\{m_\la(x\|a) \mid \la_1 < n\}$ and 
$\{\hh_\la(y\|a)\mid \la_1 < n\}$.
\end{prop}
\begin{proof} 
The duality is clear from the definition of $m_\la(x\|a)$.

Since $I_n$ is generated by primitive elements
it follows that $\La^{(n)}(x\|a)$ has an induced $S$-Hopf algebra structure.  It remains to show that $\{m_\la(x\|a) \mid \la_1 < n\}$ is indeed a $S$-basis for $\La^{(n)}(x\|a)$.  The following hold.
\begin{enumerate}
\item
$\{\mt_\la(x\|a)\mid \la\in\Y\}$ is a $\Q[a]$-basis of $\La(x\|a)$
and 
$\{1\otimes \mt_\la(x\|a)\mid \la\in\Y\}$ is an $S$-basis of $S\otimes_{\Q[a]} \La(x\|a)$
\item
$\{\mt_\la(x\|a)\mid \la_1\ge n\}$ is a $\Q[a]$-basis of $I_n$ and
$\{1 \otimes \mt_\la(x\|a)\mid \la_1\ge n\}$ is an $S$-basis of $S \otimes_{\Q[a]} I_n$.
\end{enumerate}
(1) holds since the $\mt_\la(x\|a)$ are unitriangular over $\Q[a]$ with 
the $\Q[a]$-basis $\{p_\la[x-\ap]\mid \la\in\Y\}$ of $\La(x\|a)$.
(2) can be deduced from the definition of $I_n$ and (1).

By Proposition \ref{P:monomial} the $\mt_\la(x\|a)$ are triangular with the $m_\la(x\|a)$ 
with diagonal coefficients in $\Q^\times$. It follows that all of the above basis statements
remain true with $m_\la(x\|a)$ replacing $\mt_\la(x\|a)$. It follows that $\{1\otimes (\mt_\la(x\|a)+I_n)\mid \la_1<n\}$
is an $S$-basis of $\La^{(n)}(x\|a)$.
\end{proof}

\section{Affine double Stanley, affine double Schur and $k$-double Schur functions}
Throughout this section let $k=n-1$.

\subsection{Definitions}
We define the \defn{affine double Stanley symmetric functions} $\tF_w(x\|a) \in \La^{(n)}(x\|a)$ by the following
equation in $\La^{(n)}(x\|a) \otimes_S \A$:
\begin{equation}\label{E:defneqaffineStanley}
\sum_{w\in\tS_n} \tF_w(x\|a) A_w = \sum_{\mu_1<n}  m_\la(x\|a) j_{\mu_1} j_{\mu_2}\dotsm.
\end{equation}
Restricting to Grassmannian permutations $w\in\tS_n^0$ 
we obtain the sub-family of \defn{affine double Schur functions} 
\begin{align}\label{E:equivaffSchur}
\tF_\la(x\|a):=\tF_{w_\la^\af}(x\|a)\qquad\text{for $\la_1<n$.}
\end{align}

\begin{remark}
The $\tF_w(x\|a)$ defined in \eqref{E:defneqaffineStanley} lie in the quotient $\Lambda^{(n)}(x\|a)$.  In fact, one obtains a distinguished polynomial representative, since $\{m_\lambda(x\|a) \mid \lambda_1 < n\}$ form a basis. 
\end{remark}
Taking the coefficient of $\tF_\la(x\|a)$ in \eqref{E:defneqaffineStanley}
we obtain two equivalent definitions of the
\defn{equivariant $k$-Kostka} matrix $(\hK^{(k)}_{\la\mu}(a)\mid \la_1,\mu_1<n)$.
\begin{align}\label{E:eqkKostka}
  j_{\mu_1} j_{\mu_2}\dotsm &= \sum_{\la_1<n} \hK^{(k)}_{\la\mu}(a) j_\la \\
  \tF_\la(x\|a) &= \sum_{\mu_1<n} \hK^{(k)}_{\la\mu}(a) m_\mu(x\|a).
\end{align}

\begin{remark} An explicit positive expression for the Pieri coefficients
giving the Schubert expansion of $j_r j_\nu$ has been obtained in \cite{LS:Pieri}.
Iterating the equivariant Pieri rule gives an interpretation for the equivariant
$k$-Kostka matrix. However to make the definition of $\tF_w(x\|a)$ completely explicit combinatorially, 
one also requires an explicit formula for the double monomial basis $m_\la(x\|a)$, which we do not have.
\end{remark}

\begin{prop}
The set $\{\tF_\la(x\|a)\mid \la_1<n \}$ forms a basis of $\Lambda^{(n)}(x\|a)$ over $S$.
\end{prop}
\begin{proof} By Proposition \ref{P:mbasis} it suffices to show that the equivariant $k$-Kostka matrix is
unitriangular for a suitable order on partitions.
The affine nilHecke ring is graded with $\deg(A_i) = 1$ and $\deg(a_i) = -1$.  It follows that the coefficient of $A_w$ in $j_\nu =  j_{\nu_1} j_{\nu_2}\dotsm$ is a polynomial of degree $\ell(w) - |\nu|$ in $S$.  If $|\nu| = \ell(w)$, then this coefficient is simply the coefficient of $m_\nu$ in the affine Stanley symmetric function defined in \cite{LamAS}.
Thus it follows from the triangularity results for affine Schur functions that
$$
\tF_\lambda(x\|a) = m_\lambda(x\|a) + \sum_{\mu \prec \la}\hK^{(k)}_{\la\mu}(a)\,m_\mu(x\|a) + \sum_{|\nu| <|\la|} \hK^{(k)}_{\la\nu}(a)\, m_\nu(x\|a).
$$
Here $\prec$ denotes dominance order on partitions of the same size, and one has $\hK^{(k)}_{\la\mu}(a) \in \Z$ and 
$\hK^{(k)}_{\la\nu}(a) \in S$ with degree $|\la|-|\nu|$.
\end{proof}

Define the \defn{$k$-double Schur functions} $\{s^{(k)}_\la(y\|a) \mid \la_1<n\} \subset \hL'_{(n)}(y\|a)$ to be dual to the basis $\{\tF_\la(x\|a)\mid \la_1<n \} \subset \La^{(n)}(x\|a)$:
$$
\ip{\tF_\la(x\|a)}{s^{(k)}_\mu(y\|a)}= \delta_{\la\mu}.
$$
Such elements $s^{(k)}_\mu(y\|a)$ exist and are unique because of the triangularity of $\tF_\la(x\|a)$ with respect to the $m_\lambda(x\|a)$ basis.  The $k$-double Schur functions $s^{(k)}_\mu(y\|a)$ are infinite linear combinations of the $\hh_\nu(y\|a)$, the lowest degree term of which is the image of the usual ungraded $k$-Schur function under the substitution $h_i \mapsto \hh_i(y\|a)$.

\begin{prop}
Setting $a_i = 0$ in $\tF_w(x\|a)$ and $s^{(k)}_\lambda(y\|a)$ recovers the usual affine Stanley symmetric function $\tF_w(x)$ and (ungraded) $k$-Schur function $s^{(k)}_\la(y)$ respectively.
\end{prop}
\begin{proof}
Since $m_\la(x\|a)|_{a_i = 0}$ is the usual monomial symmetric function $m_\la(x)$, and $\hh_i(y\|a)|_{a_i = 0}$ is the usual homogeneous symmetric function $h_i(y)$, this follows immediate by comparing with the definitions \cite{LamAS}.
\end{proof}

By Proposition \ref{P:generators} there is a well-defined $S$-Hopf morphism
\begin{align}\label{E:hommap}
  \La_{(n)}(y\|a)&\overset{\hommap'}\to\B \\ 
\hh_i(y\|a)&\mapsto  j_i\qquad\text{for $1\le i\le n-1$.}
\end{align}
The image $\hommap'(\La_{(n)}(y\|a))$ is not the whole of $\B$, as the following example shows.

\begin{example}\label{X:specialnongenerators}
Let $n=2$ and $\alpha=\alpha_1$. We have $j_{\id}=1$ and for $k\ge1$,
\begin{align}\label{E:n=2j}
  j_{s_0(s_1s_0)^{k-1}} &= A_{s_0 (s_1s_0)^{k-1}} + A_{s_1 (s_0s_1)^{k-1}} - \alpha A_{(s_0s_1)^k} \\
  j_{(s_1s_0)^k} &= A_{(s_1s_0)^k} + A_{(s_0s_1)^k}.
\end{align}
We deduce that $j_{s_0} j_{s_1s_0} = j_{s_0s_1s_0}$,
$j_{s_0}^2 = j_{s_1s_0} - \alpha j_{s_0s_1s_0}=j_{s_1s_0} (1 - \alpha j_{s_0})$. Therefore
$j_{s_1s_0} = j_{s_0}^2(1-\alpha j_{s_0})^{-1} = j_{s_0}^2 + \alpha j_{s_0}^3 + \alpha^2 j_{s_0}^4 +\dotsm$.
\end{example}

Let $\hB=\prod_{w\in\tS_n^0} S j_w.$ One checks in a straightfoward manner that $\hB$ is a ring.  The map $\hommap'$ extends to a ring homomorphism $\hL'_{(n)}(y\|a) \to \hB$, which we still denote by $\hommap'$.  We define $\hL_{(n)}(y\|a) \subset \hL'_{(n)}(y\|a)$ to be the $S$-span of the elements $s^{(k)}_\la(y\|a)$.  We shall presently show that $\hL_{(n)}(y\|a)$ is a ring, and that $\hommap'$ gives an isomorphism $\hL_{(n)}(y\|a) \simeq \B$.

\subsection{Main theorem}

Denote by $\hommap$ the composition of $\hommap'$ with the isomorphism $\B \simeq H_T(\Gr)$.  Thus $\hommap: \hL'_{(n)}(y\|a) \to \hat H_T(\Gr)$, where $\hat H_T(\Gr)=\prod_{w\in \tS_n^0} S [X_w]_T$.

\begin{theorem}\label{T:main}\
\begin{enumerate}
\item
There are dual Hopf-isomorphisms 
\begin{align*}
\hommap: &\hL_{(n)}(y\|a) \to H_T(\Gr) \\
\hommap^*:& H^T(\Gr) \to \Lambda^{(n)}(x\|a)
\end{align*}
such that
\begin{align*}
s^{(k)}_\la(y\|a) &\longmapsto [X_{w_\la^\af}]_T \\
[X_{w_\la^\af}]^T &\longmapsto \tF_\la(x\|a).
\end{align*}
\item
We have $\hommap'(s^{(k)}_\la(y\|a)) = j_{w_\la^\af}$.
\item
We have $\epsilon_\Gr(\tF_\la(x\|a)) = \xi^{w_\la^\af}$. 
\end{enumerate}
\end{theorem}

\begin{proof}
We first establish (2).  By general properties of dual ``bases'', see for example \cite[Proposition 5.2]{Mol}, one has the identity 
\begin{equation}
\sum_{ \lambda_1 < n}  m_\la(x\|a) \hh_\la(y\|a) = \sum_{\lambda_1 < n}
\tF_\la(x\|a) s^{(k)}_\la(y\|a)
\end{equation}
in a suitable completion of $\La^{(n)}(x\|a) \otimes \La_{(n)}(y\|a)$.  Applying $\hommap'$ to both sides, one obtains
\begin{equation}
\sum_{\la \mid \lambda_1 < n}  m_\la(x\|a) j_{\la_1}(y\|a)j_{\la_2}(y\|a) \cdots j_{\la_\ell}(y\|a) = \sum_{\lambda_1 < n}
\tF_\la(x\|a) \hommap'(s^{(k)}_\la(y\|a)).
\end{equation}
Comparing this with the definition of $\tF_\la(x\|a)$, and using the fact that $\{\tF_\la(x\|a)\}$ is linearly independent, we see that the coefficient of $A_{w^\af_\la}$ in $\hommap'(s^{(k)}_\la(y\|a))$ is equal to 1, and that no other $A_{w^\af_\mu}$ occurs.  By Theorem \ref{T:Pet}(1), we obtain (2).  By Theorem \ref{T:Pet}(2), and Hopf duality, we obtain (1).  

Finally, to check (3), we have to check that the isomorphism $\epsilon_\Gr$ is indeed $(\kappa^*)^{-1}$ composed with the isomorphism of Theorem \ref{T:GKMaff}.  But this follows from (1), Lemma \ref{L:smallclasses}, and Proposition \ref{P:cohomgen}.
\end{proof}

\begin{cor}
Suppose $\la_1+\ell(\la)\le n$. Then  $\tF_\la(x\|a) = s_\la(x\|a)$.  In particular, we have $\tF_i(x\|a) = h_i(x\|a)$ for $1 \leq i \leq n -1$.
\end{cor}
\begin{proof}
Follows immediately from Theorem \ref{T:main} and Lemma \ref{L:smallclasses}.
\end{proof}

\begin{remark} Let $d_{\la\mu}^\nu\in S$ be defined by $$s^{(k)}_\la(y\|a)s^{(k)}_{\mu}(y\|a)=\sum_\nu d_{\la\mu}^\nu s^{(k)}_\nu(y\|a).$$
These structure constants satisfy Graham positivity:
$(-1)^{|\nu|-|\la|-|\mu|} d_{\la\mu}^\nu\in \Z_{\ge0}[\alpha_1,\alpha_2,\dotsc,\alpha_{n-1}]$.
This follows from Theorem \ref{T:main}, the connection between $H_T(\Gr)$ and the equivariant quantum cohomology ring $QH^T(G/B)$ of the flag manifold \cite{P,LS:QH}, and Mihalcea's proof \cite{Mih} of positivity for the structure constants of $QH^T(G/B)$.

This is equivalent to the Graham positivity $(-1)^{\ell(x)-\ell(w)}j_w^x\in \Z_{\ge0}[\al_1,\dotsc,\al_{n-1}]$
where $w\in \tS_n^0$ and $x\in \tS_n$ with $j_w^x$ as in Theorem \ref{T:Pet}.
\end{remark}

\subsection{Affine double Stanleys and the affine flag variety}
The affine flag variety $\Fl$ of $SL_n$ is weak homotopy-equivalent to the quotient $LSU(n)/T_{\R}$, where $LSU(n)$ is the (non-based) loop space, and $T_{\R}$ is the (compact) torus of $SU(n)$.  The inclusion $\Omega SU(n) \to LSU(n)$ and the quotient map $LSU(n) \to LSU(n)/T_{\R}$ induce a pullback-map
$$
p^*: H^T(\Fl) \to H^T(\Gr).
$$
For $w \in \tS_n$, let $[X^\Fl_x]^T$ denote the equivariant cohomology Schubert classes of $H^T(\Fl)$.  It is known that the coefficient $(-1)^{\ell(x)-\ell(w)}j_w^x$ of Theorem \ref{T:Pet} is equal to the coefficient of $[X_w]^T$ in $p^*([X^\Fl_x]^T)$;see \cite[Section 5.1]{LSS} where the completely analogous situation for equivariant $K$-theory is discussed.

\begin{theorem}
Affine double Stanley symmetric functions are pullbacks of equivariant Schubert classes of the affine flag variety.  More precisely, under the isomorphism, $\kappa^*:H^T(\Gr) \simeq \La^{(n)}(x\|a)$, we have $\kappa^*(p^*([X^\Fl_x]^T)) = \tF_x(x\|a)$.
\end{theorem}
\begin{proof}
It follows from the proof of Theorem \ref{T:main} that the coefficient of $\tF_\la(x\|a)$ in $\tF_x(x\|a)$ is $(-1)^{\ell(x)-\ell(w)}j_w^x$.
\end{proof}

\subsection{Embeddings of affine Grassmannians}
For a discussion of non-equivariant branching coefficients we refer the reader to \cite{LLMS} for a combinatorial treatment, and \cite{Lam11} for a geometric explanation.

Recall from \S \ref{SS:Gr} that there is a $T$-equivariant embedding $\iota_\infty:\Gr\to \Gr_\infty$.
There is therefore a $\Q[a]$-algebra homomorphism $H^{T_\Z}(\Gr_\infty)\to H^T(\Gr_\infty)\to H^T(\Gr)$,
where $\Q[a]$ acts via the forgetful map $\For:\Q[a]\to S$ on the latter two $S$-algebras.  On the level of $T$-fixed points the embedding $\iota_\infty$ sends an $n$-core to itself, treated as a partition.  
By Theorems \ref{T:GKM} and \ref{T:GKMaff} the induced homomorphism corresponds to the projection map $\pi_\infty:\La(x\|a)\to \La^{(n)}(x\|a)$ which one observes is a Hopf map.

Let $\la,\nu\in\Y$ with $\la_1<n$. Define the equivariant branching coefficients $b_{\la,\nu}^{(k,\infty)}\in S$ by
\begin{align*}
  \pi_\infty(s_\nu(x\|a)) = \sum_{\la_1<n} b_{\la\nu}^{(k,\infty)} \tF_\la(x\|a).
\end{align*}
By Hopf duality, one also has
\begin{align*}
\s^{(k)}_\la (y\|a) = \sum \sum_{\nu} b_{\la\nu}^{(k,\infty)} \hs_\la(y\|a).
\end{align*}

These coefficients may be related to the equivariant $k$-Kostka coefficients. Define the
equivariant Kostka coefficients $\hK_{\la\mu}(a)\in\Q[a]$ by either of the following:
\begin{align}
  \hh_\mu(y\|a) &= \sum_\la \hK_{\la\mu}(a) \hs_\la(y\|a) \\
  s_\la(x\|a) &= \sum_\mu \hK_{\la\mu}(a) m_\mu(x\|a).
\end{align}
By the various definitions we obtain, for $\mu,\nu\in\Y$ with $\mu_1<n$,
\begin{align}
  \For(\hK_{\nu\mu}(a)) &= \sum_{\la_1<n} b_{\la\nu}^{(k,\infty)} \hK^{(k)}_{\la\mu}(a).
\end{align}
The equivariant $k$-Kostka matrix can be inverted over $S$ to yield
an expression for the branching coefficients.

More generally, let $m = rn$ for $r > 1$.  Comparing the embeddings $\Gr_{SL_n} \hookrightarrow \Gr_\infty$ and $\Gr_{SL_m} \hookrightarrow \Gr_\infty$ we see that one has a $T_n$-equivariant embedding $\Gr_{SL_n} \to \Gr_{SL_m}$, where $T_n$ is the maximal torus in $SL_n$.  On the level of $T_n$-fixed points,  we are treating $n$-cores as $m = rn$-cores.  On the level of symmetric functions we are considering the projection map $\pi:\La^{(m)}(x\|a)\to \La^{(n)}(x\|a)$.  We define the equivariant branching coefficients $b_{\la,\nu}^{(n-1,m-1)}\in S$ by
\begin{align*}
  \pi(\tF^{(m)}_\nu(x\|a)) = \sum_{\la_1<n} b_{\la\nu}^{(n-1,m-1)} \tF^{(n)}_\la(x\|a)
\end{align*}
where we have used the superscripts to distinguish the affine double Schur functions for different affine symmetric groups.

Recall that an element $a \in S$ is called Graham-positive if it is a polynomial in the simple roots with nonnegative integer coefficients.

\begin{conj} 
The equivariant branching coefficients $b_{\la\nu}^{(n-1,m-1)}$ (including $m = \infty$) are Graham-positive up to a predictable sign.  Equivalently, the image of Schubert classes under the map
$H^{T_n}(\Gr_{SL_{rn}})\to H^{T_n}(\Gr_{SL_n})$ are effective classes.
\end{conj}

For the graded non-equivariant $k$-Schur functions the branching positivity conjecture was made in \cite{LLM}.  For proofs in the case of ungraded $k$-Schur functions see \cite{LLMS, Lam11}.

\section{Tables and examples}

\subsection{Double monomial functions}
Since $p_r[x-\ap]=m_r[x-\ap]$ the expansion
of $m_\la(x\|a)$ for $\la=(r)$ is given by Proposition \ref{P:monomial}.
We write $m_\la$ for $m_\la[x-\ap]$, which is the basis of $\La(x\|a)$ dual to 
the usual homogeneous basis $\{h_\la(y)\}\subset \hL(y\|a)$. 
\begin{align*}
  m_{11}(x\|a) &= m_{11} + a_1 m_1 \\
  m_{21}(x\|a) &= m_{21} + a_1 m_2 - 2(a_0+a_1)m_{11} - a_1(2a_0+a_1)m_1 \\
  m_{111}(x\|a) &= m_{111}+2a_1 m_{11} + a_1^2 m_1 \\
  m_{31}(x\|a) &= m_{31}+a_1m_3-(a_{-1}+a_0+a_1)m_{21}-a_1(a_{-1}+a_0+a_1)m_2 \\
  &+2(a_{-1}a_0+a_{-1}a_1+a_0a_1)m_{11}+
  a_1(2a_{-1}a_0+a_{-1}a_1+a_0a_1) m_1 \\
  m_{22}(x\|a) &= m_{22}-(a_0+a_1)m_{21}-a_0a_1m_2+(a_0+a_1)^2m_{11}+a_0a_1(a_0+a_1)m_1 \\
  m_{211}(x\|a) &= m_{211} +2a_1 m_{21}+a_1^2m_2 - 3(a_0+a_1)m_{111} -2a_1(3a_0+2a_1)m_{11}-a_1^2(3a_0+a_1)m_1 \\
  m_{1111}(x\|a) &= m_{1111}+3a_1m_{111}+3a_1^2m_{11}+a_1^3m_1.
\end{align*}
Observe that $$m_{1^r}(x\|a) = \sum_{j=1}^r \binom{r-1}{j-1} a_1^{r-j} m_{1^j}\qquad\text{for $r\ge1$.}$$

\subsection{Equivariant $k$-Kostka matrix}
In all cases $\hK^{(k)}_{\varnothing,\mu}(a)=\delta_{\varnothing,\mu}$,
and $\hK^{(k)}_{(1),\mu}(a)=\delta_{(1),\mu}$, $\hK^{(k)}_{\la,\varnothing}(a)=\delta_{\la,\varnothing}$
and $\hK^{(k)}_{\la,(1)}(a)=\delta_{\la,(1)}$.

\subsubsection{$n=2$ ($k=1$)}
\begin{align*}
  \hK^{(1)}_{(1^p),(1^{p-j})}(a) = \binom{\lfloor (p-1)/2\rfloor}{j} (-\al_1)^j.
\end{align*}

\subsubsection{$n=3$ ($k=2$)}
\begin{align*}
\begin{array}{|c|||c|c|c|c|c|c|c|c|c|c|}\hline\
&11&2&111&21&1^4&211&22&1^5&21^3&221\\ \hline\hline
11 & 1&&&&&&&&& \\ \hline
2 &1&1&&&&&&&& \\ \hline
111&-\al_2&&1&&&&&&& \\ \hline
21&-\al_1&&1&1&&&&&&\\ \hline
1^4 &\al_2(\al_1+\al_2)&&-\al_1-2\al_2&&1&&&&&\\ \hline
211&\al_1\al_2&&-2(\al_1+\al_2)&-\al_2&2&1&&&&\\ \hline
22&\al_1(\al_1+\al_2)&&-2\al_1-\al_2&-\al_1-\al_2&1&1&1&&&\\ \hline
1^5&&&\al_2(\al_1+\al_2)&&-\al_1-2\al_2&&&1&&\\ \hline
21^3&-\al_1\al_2(\al_1+\al_2)&&2\al_1^2+5\al_1\al_2+2\al_2^2&\al_2(\al_1+\al_2)&-4(\al_1+\al_2)&-\al_1-2\al_2&&2&1&\\ \hline
221&&&\al_1(\al_1+\al_2)&&-2\al_1-\al_2&-\al_1-\al_2&-\al_2&1&1&1\\ \hline
\end{array}
\end{align*}

\subsubsection{$n=4$ ($k=3$)}

\begin{align*}
\begin{array}{|c||c|c|c|c|c|c|c|c|c|} \hline
& 11 & 2 & 111 & 21 &3& 1^4 & 211 & 22&31 \\ \hline\hline
11 & 1 & & & & & &&& \\ \hline
2 & 1 & 1 & & & && && \\ \hline
111& -\al_1 && 1 & & && && \\ \hline
21&-\al_1-\al_2-\al_3&& 2&1&&&&& \\ \hline
3 & -\al_3 & & 1 & 1 & 1& &&& \\ \hline
1^4 & \al_1(\al_1+\al_2)&&-2\al_1-\al_2&&&1&&&\\ \hline
211&\al_1\al_3&&-2\al_1-\al_2-2\al_3&-\al_1&&2&1&& \\ \hline
22& (\al_1+\al_2+\al_3)^2&&-3(\al_1+\al_2+\al_3)&-\al_1-\al_2-\al_3&&2&1&1& \\ \hline
31 & \al_3(\al_2+\al_3)&&-\al_2-2\al_3&-\al_2-\al_3&&1&1&1&1 \\ \hline
\end{array}
\end{align*}

\subsubsection{$n=\infty$ ($k=\infty$)}
\begin{align*}
\begin{array}{|c||c|c|c|c|c|c|c|c|c|c|} \hline 
&11&2&1^3&21&3&1^4&211&22&31&4\\ \hline \hline
11  &1&&&&&&&&& \\ \hline
2   &1&1&&&&&&&& \\ \hline
1^3 &-a_1+a_2& &1&&&&&&& \\ \hline
21  &a_0-a_1& &2&1&&&&&& \\ \hline
3   &-a_{-1}+a_0& &1&1&1&&&&& \\ \hline
1^4 &(a_1-a_2)(a_1-a_3)& &-2a_1+a_2+a_3& & &1&&&& \\ \hline
211 &-(a_0-a_1)(a_1-a_2)&&2a_0-4a_1+2a_2&-a_1+a_2 & &3 &1&&& \\ \hline
22  &(a_0-a_1)^2& &3a_0-3a_1 &a_0-a_1 & &2 &1 &1&& \\ \hline
31  &-(a_{-1}-a_0)(a_0-a_1)&&-2a_{-1}+4a_0-2a_1&a_0-a_1 & &3 &2 &1 &1& \\ \hline
4   &(a_{-2}-a_0)(a_{-1}-a_0) & &-a_{-2}-a_{-1}+2a_0 &-a_{-2}+a_0 & &1 &1 &1 &1 &1 \\ \hline
\end{array}
\end{align*}

\subsection{Branching coefficients}
For $m=rn$ a positive integer multiple of $n$ (or $m=\infty$) there is a canonical projection
$\La^{(m)}(x\|a) \to \La^{(n)}(x\|a)$. It may be computed by taking the double monomial
expansion, removing all terms except those of $m_\mu(x\|a)$ for $\mu_1<n$,
and then applying to the coefficients, the forgetful map
$\Q[a_1,\dotsc,a_m]\to\Q[a_1,\dotsc,a_n]$ with $a_{i+kn}\mapsto a_i$
for $1\le i\le n$.

For $\la$ with $\la_1<m$ let $F_{w_\la}^{(m)}$ denote the affine double Schur
function in $\La^{(m)}(x\|a)$; for $m=\infty$ this is a double Schur function.
We compute the affine double Schur expansion of its image in $\La^{(n)}(x\|a)$
as follows. The double monomial expansion of $F_{w_\la}^{(m)}$ in $\La^{(m)}(x\|a)$ is given by
the equivariant $(m-1)$-Kostka matrix. Examples were computed previously.
We then reduce the coefficients of the matrix by the forgetful map and keep columns
corresponding to the coefficients of $m_\mu(x\|a)$ for $\mu_1<n$.
Call the result the reduced Kostka matrix. It is the first of two matrices given in the
data below.

Consider the reduced Kostka matrix. Its rows indexed by $\la$ with $\la_1<n$,
are the monomial expansion of the Schubert basis of the target space $\La^{(n)}(x\|a)$
(affine double Schurs for $\Gr_{SL_n}$).
The second matrix in the following data, expresses the rows indexed by $\la$ for $\la_1\ge n$,
in terms of the aforementioned basis rows, yielding the required expansion.

We use $m=\infty$ unless stated otherwise.

\subsubsection{$n=2$ ($k=1$)}
\begin{align*}
\begin{array}{|c||c|c|c|c|c|} \hline 
&\varnothing&1&11&1^3&1^4\\ \hline \hline
\varnothing&1&&&&  \\ \hline
 1  &&1&&& \\ \hline
11  &&&1&& \\ \hline
2   &&&1&& \\ \hline
1^3 &&&-\al_1 &1& \\ \hline
21  &&&-\al_1&2& \\ \hline
3   &&&-\al_1 &1& \\ \hline
1^4 &&&&-\al_1&1 \\ \hline
211 &&&\al_1^2&-4\al_1&3  \\ \hline
22  &&&\al_1^2&-3\al_1&2  \\ \hline
31  &&&\al_1^2&-4\al_1& 3\\ \hline
4   &&&  &-\al_1 &1  \\ \hline
\end{array}
\end{align*}
\begin{align*}
\begin{array}{|c||c|c|c|c|c|} \hline 
&11&1^3&1^4\\ \hline \hline
2   &1&& \\ \hline
21  &\al_1&2& \\ \hline
3   &&1& \\ \hline
211 &&-\al_1&3 \\ \hline
22  &&-\al_1 &2  \\ \hline
31  &&-\al_1 &3 \\ \hline
4   & &&1 \\ \hline
\end{array}
\end{align*}

\subsubsection{$n=3$ ($k=2$)}
\begin{align*}
\begin{array}{|c||c|c|c|c|c|c|c|c|c|c|} \hline 
&\varnothing&1&11&2&1^3&21&1^4&211&22\\ \hline \hline
\varnothing&1&&&&&&&& \\ \hline
 1  &&1&&&&&&& \\ \hline
11  &&&1&&&&&& \\ \hline
2   &&&1&1&&&&& \\ \hline
1^3 &&&-\alpha_2& &1&&&& \\ \hline
21  &&&-\alpha_1-\alpha_2& &2&1&&& \\ \hline
3   &&&-\alpha_1& &1&1&&& \\ \hline
1^4 &&&\alpha_2(\alpha_1+\alpha_2)& &-\al_1-2\al_2&&1&& \\ \hline
211 &&&\al_2(\al_1+\al_2)&&-2\alpha_1-4\alpha_2&-\al_2&3 &1& \\ \hline
22  &&&(\al_1+\al_2)^2& &-3\al_1-3\al_2&-\al_1-\al_2&2 &1 &1 \\ \hline
31  &&&\al_1(\al_1+\al_2)&&-4\al_1-2\al_2&-\al_1-\al_2&3 &2 &1  \\ \hline
4   &&&\al_1(\al_1+\al_2)& &-2\al_1-\al_2&-\al_1-\al_2&1 &1 &1  \\ \hline
\end{array}
\end{align*}
\begin{align*}
\begin{array}{|c||c|c|c|c|c|c|} \hline 
&1^3&21&1^4&211&22\\ \hline \hline
3   &-1&1&&& \\ \hline
31  &-\al_1-\al_2&\al_2&-2 &1 &1  \\ \hline
4    &&&-1 & &1  \\ \hline
\end{array}
\end{align*}

\subsubsection{$m=4$ and $n=2$}
\begin{align*}
\begin{array}{|c||c|c|c|} \hline
& 11  & 111  & 1^4 \\ \hline\hline
11 & 1 &   &  \\ \hline
2 & 1 & &\\ \hline
111& -\al_1 &1  & \\ \hline
21&-\al_1& 2& \\ \hline
3 & -\al_1 & 1  & \\ \hline
1^4 & &-\al_1&1 \\ \hline
211&\al_1^2&-3\al_1&2 \\ \hline
22& \al_1^2&-3\al_1&2 \\ \hline
31 & &-\al_1&1 \\ \hline
\end{array}
\end{align*}

\begin{align*}
\begin{array}{|c||c|c|c|} \hline
& 11  & 111  & 1^4 \\ \hline\hline
2 & 1 & &\\ \hline
21&\al_1& 2& \\ \hline
3 &  & 1  & \\ \hline
211&&-\al_1&2 \\ \hline
22&&-\al_1&2 \\ \hline
31 & &&1 \\ \hline
\end{array}
\end{align*}

\end{document}